\documentclass[reqno]{amsart}
\usepackage{amssymb}
\usepackage{amscd}

\newtheorem{theorem}{Theorem}[section]
\newtheorem{corollary}[theorem]{Corollary}

\newtheorem{proposition}[theorem]{Proposition}

\newtheorem{remark}{Remark}[section]

\theoremstyle{definition}
\theoremstyle{remark}
\numberwithin{equation}{section}

\theoremstyle{plain}

\begin{document}
\title[BT-Theorem and automatic continuity]{Quantitative BT-Theorem and
automatic continuity for standard von Neumann algebras}
\author{Francesco Fidaleo}
\address{Department of Mathematics, University of Rome ``Tor Vergata'',
Via della Ricerca Scientifica, 00133 Rome, Italy}
\email{fidaleo@axp.mat.uniroma2.it}
\author{L\'aszl\'o Zsid\'o}
\address{Department of Mathematics, University of Rome ``Tor Vergata'',
Via della Ricerca Scientifica, 00133 Rome, Italy}
\email{zsido@axp.mat.uniroma2.it}
\date{May 12, 2015}
\renewcommand{\subjclassname}{\textup{2000} 
 Mathematics Subject Classification}
\thanks{The authors were supported by MIUR, INDAM and EU}
\dedicatory{Dedicated to Professor G. A. Elliott on his
$\, 70^{\text{th}}$ birthday}

\begin{abstract}
We prove a general criterion for a von Neumann algebra $M$ in order to be in standard form.
It is formulated in terms of an everywhere defined, invertible, antilinear, a priori not
necessarily bounded operator, intertwining $M$ with its commutant $M'$ and acting
as the $*$-operation on the centre. We also prove a generalized version of the BT-Theorem
which enables us to see that such an intertwiner must be necessarily bounded.
It is shown that this extension of the BT-Theorem leads to the automatic boundedness of
quite general operators which intertwine the identity map of a von Neumann algebra
with a general bounded, real linear, operator valued map. We apply the last result to the
automatic boundedness of linear operators implementing algebraic morphisms of a
von Neumann algebra onto some Banach algebra, and to the structure of a $W^*$-algebra
$M$ endowed with a normal, semi-finite,
faithful weight $\varphi\,$, whose left ideal $\mathfrak N_{\varphi}$ admits an algebraic complement in the GNS representation space $H_{\varphi}\,$, invariant under the
canonical action of $M$.

\end{abstract}
\maketitle

\section{Introduction}

A von Neumann algebra $M$ on a Hilbert space $H$ is usually called {\it standard}
if there exists a bijective isometrical antilinear involution (called {\it conjugation})
$J : H\longrightarrow H$ such that the mapping $x\longmapsto Jx^*J$ is a
$*$-anti-isomorphism of $M$ onto its commutant $M'$, acting as the $*$-operation on
the centre $Z(M)$ of $M$. In other words, $J$ should satisfy the conditions
\smallskip

\centerline{$JMJ=M'\text{ and }JzJ=z^*,\qquad z\in Z(M)\, .$}
\smallskip

\noindent Any two $*$-isomorphic standard von Neumann algebras are spatially
isomorphic (see e.g. \cite{Dix}, Chapitre III, \S 1, Th\'eor\`eme 6 or \cite{S-Z}, Corollary 10.15).

On the other hand, the Tomita-Takesaki Theory (\cite{T1}) enabled
the construction of a standard representation of every von Neumann algebra. Indeed,
each von Neumann algebra $M$ has a normal semi-finite faithful (n.s.f. for short) weight
$\varphi$, and then the associated GNS representation
$\pi_\varphi : M\longrightarrow B(H_\varphi )$ yields a $*$-isomorphism of $M$ onto the
von Neumann algebra $\pi_\varphi (M)\,$, which is standard because the modular
conjugation $J_\varphi$ corresponding to $\varphi$ satisfies
\smallskip

\centerline{$J_\varphi \pi_\varphi (M)J_\varphi =\pi_\varphi (M)'\text{ and }J_\varphi zJ_\varphi
=z^*,\qquad z\in Z\big(\pi_\varphi (M)\big)$}
\smallskip

\noindent (see e.g. \cite{S-Z}, 10.14).

We recall that a $\sigma$-finite von Neumann algebra is standard if and only if
it has a cyclic and separating vector (see e.g. \cite{S-Z}, 10.6, Corollary 1 in 10.13
and Theorem 10.25).

In conclusion, every von Neumann algebra has a standard representation which
is unique up to spatial isomorphism. A more refined classification of the possible
standard representations was elaborated by H. Araki (\cite{A}) and A. Connes (\cite{Co2})
for von Neumann algebras having a cyclic and separating vector, and by U. Haagerup
(\cite{Haa1}) in the general case.

We shall prove a general criterion of standardness, namely that a von Neumann algebra
$M$ on a Hilbert space is standard whenever there exists a bijective antilinear operator
$T : H\longrightarrow H$ such that
\smallskip

\centerline{$TMT^{-1}=M'\text{ and }\;\! TzT^{-1}=z^*,z\in Z(M)$}
\smallskip

\noindent (Theorem \ref{standardness}). This theorem will be used in a forthcoming paper
on tensor products of von Neumann algebras over von Neumann subalgebras.

Furthermore, we shall prove that the above antilinear operator $T$, not a priori assumed
to be bounded, is necessarily bounded (Proposition \ref{boundedness}). This boundedness
result, which can be proved by using a result of E. L. Griffin (see \cite{Gr2}, Theorems 1 and 2),
arises also as particular case of a general automatic continuity theorem
(Theorem \ref{commutant}), whose proof cannot be carried out by applying the results of
Griffin. Its proof is based on a generalization of the classical BT-Theorem (Theorem \ref{BT}),
which could be of interest also elsewhere as it is shown in Section \ref{applx} containing
some useful applications.

We shall use the terminology of \cite{S-Z}. In particular,
\begin{itemize}
\item $(\, \cdot\, |\, \cdot\, )$ will denote the inner product of a Hilbert space and it will
be assumed linear in the first variable and antilinear in the second variable;
\item $B(H)$ will denote the algebra of all bounded linear operators on the Hilbert space $H$,
with the identity simply denoted by $1$;
\item $Z(M)$ will denote the centre of a von Neumann algebra $M\subset B(H)\,$;
\item $l(x)$ and $r(x)$ will stay for the left and right support-projection of an operator $x$
in some von Neumann algebra $M\subset B(H)\,$, where $l(x)=r(x)=:s(x)$ if $x$ is normal,
$z(x)$ for the central support projection of $x\in M$, and $s(\varphi )$ for the support projection
of a normal positive linear functional $\varphi$ on $M$.
\end{itemize}

\medskip
\section{A general criterion for standardness of von Neumann algebras}
\bigskip

We say that a projection $e$ in a von Neumann algebra $M$ is
{\it piecewise $\sigma$-finite} if there exists a family $\big( p_\iota\big)_{\iota\in I}$ of
mutually orthogonal central projections of $M$ such that $\displaystyle \sum\limits_{\iota\in I} p_\iota =1$
and all projections $e p_\iota$ are $\sigma$-finite. If the unit of $M$ is piecewise
$\sigma$-finite, then we call the von Neumann algebra $M$ piecewise $\sigma$-finite.

Let us first consider the characterization of the standardness of piecewise $\sigma$-finite
von Neumann algebras.

\begin{proposition}\label{piecewise-sigma-finite}
Let $M\subset B(H)$ be a piecewise $\sigma$-finite von Neumann algebra.
If there is a bijective antilinear operator $T : H\longrightarrow H$ such that
\medskip

\centerline{$TMT^{-1}=M'\text{ and }\;\! TzT^{-1}=z^*,\qquad z\in Z(M)\, ,$}
\smallskip

\noindent then $M$ is a standard von Neumann algebra.
\end{proposition}

\proof
Since $T$ is commuting with the central projections of $M$ and direct sums of standard
von Neumann algebras are standard, we may assume without loss of generality that
$M$ is $\sigma$-finite.

According to Lemma 7.18 of \cite{S-Z}, there exists a projection $p\in Z(M)$ such that
the reduced/induced algebra $M_{p} =\big\{ x\mid pH : pH\longrightarrow pH\, ; x\in M
\big\} \subset B(pH)$ has a cyclic vector $\xi\in pH\,$, while $M_{1-p}\subset B((1-p)H)$
has a separating vector $\eta\in(1-p)H$.

Then $T^{-1}\xi\in pH$ is separating for $M_{p}\,$.
Namely, suppose that $x\;\! T^{-1}\xi =0$ for some $x\in pM\,$. Then
\smallskip

\centerline{$y\;\! TxT^{-1}\xi =0\, ,\qquad y\in pM\, .$}
\medskip

\noindent This means
\smallskip

\centerline{$0=y\;\! TxT^{-1}\xi =TxT^{-1}y\;\! \xi\, ,\qquad y\in pM,$}
\medskip

\noindent and by the cyclicity of $\xi$ for $M_p$ we deduce that
\medskip

\centerline{$TxT^{-1} = TxpT^{-1}=TxT^{-1}p=0\, .$}
\medskip

\noindent Thus, $x=0\,$.

Using now the Dixmier--Mar\'echal Theorem (see \cite{Dix-Ma}, Corollaire 1),
we infer that $pH$ contains a vector which is cyclic and separating for
$M_p\,$.

Similarly, the vector $T\eta\in (1-p)H$ is cyclic for $M_{1-p}\,$.
Indeed, the orthogonal projection $e'$ onto the closure of $(1-p)MT\eta =MT(1-p)\eta =
MT\eta$ belongs to $M'$ and $e'\leq 1-p\,$. Furthermore, $e'T\eta =T\eta$ and so
$T^{-1}e'T\eta =\eta\,$, where $T^{-1}e'T\in (1-p)M$. Since $\eta$ is separating for
$M_{1-p}\,$, we get that $T^{-1}e'T =1-p\,$, hence $e'=1-p\,$. In other words
$(1-p)MT\eta$ is dense in $(1-p)H$.

Applying again the Dixmier--Mar\'echal Theorem, we infer that $(1-p)H$
contains a vector which is cyclic and separating for $M_{1-p}\,$.

We conclude that $M$ has a cyclic and separating vector and therefore it
is a standard von Neumann algebra.

\hfill $\square\quad$
\medskip

Since finite von Neumann algebras are piecewise $\sigma$-finite (see e.g.
\cite{S-Z}, Lemma 7.2), Proposition \ref{piecewise-sigma-finite} implies
immediately :

\begin{corollary}\label{finite}
Let $M\subset B(H)$ be a finite von Neumann algebra such that for an appropriate
bijective antilinear operator $T : H\longrightarrow H$ we have
\medskip

\centerline{$TMT^{-1}=M'\text{ and }\;\! TzT^{-1}=z^*,\qquad z\in Z(M)\, .$}
\smallskip

\noindent Then $M$ is a standard von Neumann algebra.
\end{corollary}

\hfill $\square\quad$
\medskip

Now we are going to characterize the standard properly infinite von Neumann algebras.

We recall that, for a given infinite cardinal $\gamma\,$, a properly infinite von Neumann algebra
$M$ is called {\it uniform of type} $\gamma$ if there exists a family $(e_\iota )_{\iota\in I}$
of equivalent, mutually orthogonal, piecewise $\sigma$-finite projections in $M$ such that
$\displaystyle \sum\limits_{\iota\in I} e_\iota =1$ and the cardinality of $I$ is $\gamma\,$.
Every properly infinite von Neumann algebra $M$ has a unique decomposition
in uniform components: there exists a set $\Gamma$ of distinct cardinals and a
family $(p_\gamma )_{\gamma\in\Gamma}$ of non-zero central projections, uniquely
determined by the conditions
\medskip

\centerline{$\displaystyle \sum\limits_{\gamma\in\Gamma} p_\gamma =
1\, ,\qquad M_{p_\gamma}\text{is uniform of type }\gamma\text{ for every }\gamma
\in\Gamma$}

\noindent (see \cite{S-Z}, proposition 8.5).

\begin{proposition}\label{prop.inf}
Let $M\subset B(H)$ be a properly infinite von Neumann algebra. If there exists
a multiplicative antilinear isomorphism $\theta : M\longrightarrow M'$, which acts
on the centre of $M$ as the $*$-operation, then $M$ is a standard von Neumann algebra.
\end{proposition}

\proof
First we reduce the proof to the case when $\theta$ is additionally a $*$-map.

Let $\varphi$ be some n.s.f. weight on $M$, $\pi_\varphi : M
\longrightarrow B(H_\varphi )$ the associated GNS representation, and $J_\varphi$
the corresponding modular conjugation. Then $\theta_\varphi : \pi_\varphi (M)\ni
\pi_\varphi (x)\longmapsto J_\varphi \pi_\varphi (x) J_\varphi\in  \pi_\varphi (M)'$
is a multiplicative antilinear isomorphism commuting with the $*$-operation.

Since $\theta_\varphi\circ\pi_\varphi\circ\theta^{-1} : M'\longrightarrow
\pi_\varphi (M)'$ is an algebra isomorphism, by Theorem I in \cite{O} there
exists an invertible $0\leq a'\in M'$ such that
\medskip

\centerline{$M'\ni x'\longmapsto (\theta_\varphi\circ\pi_\varphi\circ\theta^{-1})
(a' x' a'^{-1})\in \pi_\varphi (M)'$}
\smallskip

\noindent is a $*$-isomorphism. Consequently its composition with $(\theta_\varphi
\circ\pi_\varphi )^{-1}$, that is
\medskip

\centerline{$M'\ni x'\longmapsto \theta^{-1}(a' x' a'^{-1})\in M,$}
\smallskip

\noindent as well as the inverse map
\smallskip

\centerline{$\theta_o : M\ni x\longmapsto a'^{-1}\theta (x) a'\in M'$}
\smallskip

\noindent of this, are multiplicative antilinear isomorphisms commuting with the
$*$-operation.

Thus $\theta_o$ is completely additive, positive, and preserves
Murray-von Neumann equivalence of projections.
Moreover, $\theta_o$ acts on the centre of $M$ as the $*$-operation.

Taking now in account the decomposability of $M$ in uniform components, as
well as the fact that direct sums of standard von Neumann algebras are standard,
we can assume in the sequel without loss of generality that $M$ is uniform of
type $\gamma$ for some infinite cardinal $\gamma\,$.

Since $M'=\theta_o(M)\,$, the commutant $M'$ is uniform of type $\gamma\,$.
On the other hand, taking into account that
\smallskip

\centerline{$M\ni x\longmapsto J_\varphi\pi_\varphi (x)J_\varphi \in\pi_\varphi (M)'$}
\smallskip

\noindent is a multiplicative antilinear isomorphism commuting with the $*$-operation,
also the commutant $\pi_\varphi (M)'$ is uniform of type $\gamma\,$.

Using now a classical implementation theorem (essentially Theorem 2 of \cite{Gr1}, cf.
\cite{S-Z}, Theorem 8.6), we conclude that the $*$-isomorphism
$\pi_\varphi : M\longrightarrow \pi_\varphi (M)$ of the von Neumann algebras $M$
and $\pi_\varphi (M)\,$, whose commutants are uniform of the same type $\gamma\,$,
is spatial. Since any von Neumann algebra which is spatially isomorphic to a standard
von Neumann algebra is still standard, we infer that $M$ is a standard von Neumann
algebra.

\hfill $\square\quad$
\medskip

A direct consequence of Corollary \ref{finite} and Proposition \ref{prop.inf} is the
following general standardness criterion :

\begin{theorem}\label{standardness}
{\rm (General criterion for standardness of von Neumann algebras)}
Let $M$ be a von Neumann algebra on a Hilbert space $H$. $M$ is acting in standard form if and only if
there is a bijective
antilinear operator $T : H\longrightarrow H$ such that
\medskip

\centerline{$TMT^{-1}=M'\text{ and }\;\! TzT^{-1}=z^*,\qquad z\in Z(M)\, .$}

\end{theorem}

\hfill $\square\quad$
\medskip

\section{Linear operators commuting with a von Neumann algebra}
\bigskip

As it is shown in the Appendix, an everywhere defined, bijective, antilinear operator
on a Hilbert space, even an involutive one, might be unbounded. Nevertheless, as we
shall show in this section, this is not the case for an operator $T$ satisfying the
assumptions in Theorem \ref{standardness}.

A first proof will be based on the result of E. L. Griffin (see \cite{Gr2}, Theorems 1 and 2)
reported below:

\begin{theorem}\label{Griffin}
Let $M$ be a von Neumann algebra on a Hilbert space $H$. In order that every linear operator
$T : H\longrightarrow H$ satisfying
\smallskip

\centerline{$\quad\quad Tx=x\;\! T\, ,\qquad x\in M$}
\smallskip

\noindent be bounded, it is necessary and sufficient that no minimal projection $p$ of
$Z(M)$ exists with $pM$ finite-dimensional and $pM'$ infinite-dimensional.
\end{theorem}

\hfill $\square\quad$
\medskip

From Theorem \ref{Griffin} it follows immediately the next automatic boundedness
result:

\begin{corollary}\label{Griffin-standard}
If $M\subset B(H)$ is a standard von Neumann algebra, then every linear operator
$T : H\longrightarrow H$ satisfying
\medskip

\centerline{$\quad\quad Tx=x\;\! T\, ,\qquad x\in M$}

\noindent is bounded.
\end{corollary}

\proof
By Theorem \ref{Griffin} it is enough to verify that if $p\in Z(M)$ is a projection such that
the reduced/induced von Neumann algebra $M_p$ is finite-dimensional, then also its
commutant $(M')_p$ is finite-dimensional.

But since $M$ is standard and the projection $p$ is central, also $M_p$ is a standard
von Neumann algebra. In particular, $(M')_p$ is $*$-anti-isomorphic to $M_p$ and therefore
it is finite-dimensional.

\hfill $\square\quad$
\medskip

An alternative proof, based on an extended version of the BT-Theorem, will be
presented in Corollary \ref{commutant-cases}.

Now we are ready to prove

\begin{proposition}\label{boundedness}
If $M\subset B(H)$ is a von Neumann algebra and $T : H\longrightarrow H$ is a bijective
antilinear operator such that
\medskip

\centerline{$TMT^{-1}=M'\text{ and }\;\! TzT^{-1}=z^*,\qquad z\in Z(M)\, ,$}
\smallskip

\noindent then $T$ must be bounded.
\end{proposition}

\proof
First at all, by Theorem \ref{standardness} the von Neumann algebra $M$ is standard. Let
$J : H\longrightarrow H$ be a conjugation satisfying the conditions
\medskip

\centerline{$JMJ=M'\text{ and }JzJ=z^*,\qquad z\in Z(M)\, .$}
\smallskip

\noindent Then $JT : H\longrightarrow H$ is a bijective linear operator such that the mapping
\medskip

\centerline{$M\ni x\longmapsto JTx(JT)^{-1} =JTxT^{-1}J\in M$}
\smallskip

\noindent is an algebra automorphism.

Next, by Theorem I in \cite{O}, there exists an invertible $0\leq a\in M$ such that
\smallskip

\centerline{$M\ni x\longmapsto JTax(JTa)^{-1} =JTaxa^{-1}T^{-1}J\in M$}
\smallskip

\noindent is a $*$-isomorphism.

Finally, since every $*$-isomorphism between standard von Neumann algebras is
spatial (see e.g. \cite{Dix}, Chapitre III, \S 1, Th\'eor\`eme 6 or \cite{S-Z}, Corollary 10.15),
there exists a unitary $U\in B(H)$ for which
\smallskip

\centerline{$JTax(JTa)^{-1} =U^{-1}x\;\! U\, ,\qquad x\in M\, ,$}
\smallskip

\noindent that is
\smallskip

\centerline{$(UJTa)x =x(UJTa)\, ,\qquad x\in M\, .$}
\medskip

Now Corollary \ref{Griffin-standard} yields the boundedness of $UJTa\,$, hence
also the boundedness of $T=JU^{-1}(UJTa)a^{-1}$.

\hfill $\square\quad$
\medskip

\section{A quantitative BT-Theorem}
\bigskip

The classical "BT-Theorem" of Murray and von Neumann (see e.g. \cite{Sak2},
Theorem 2.7.14 or \cite{S-Z}, C.6.1) states that if $M$ is a von Neumann algebra
on the Hilbert space $H\,$, $\xi_o\in H$ and $\xi$ belongs to the closure of
$M\xi_o\,$, then $\xi =b\;\!T\xi_o$ where $b\in M$ and $T$ is a densely defined,
closed linear operator, affiliated to $M$. Thus, roughly speaking, we can "lift"
any vector in $\overline{M\xi_o}$ to an operator $b\;\!T$ "related" to $M$.
The more recent proof, due essentially to R. V. Kadison and presented by
C. F. Skau in \cite{Sk}, Lemma 3.4, can be extended to obtain the following
"quantitative" version of the BT-Theorem, which will allow to "lift" vector
sequences in $\overline{M\xi_o}$ converging sufficiently fast to zero in
operator sequences which converge to zero in operator norm.

\begin{theorem}\label{BT}
{\rm (Quantitative BT-Theorem)}
Let $M$ be a von Neumann algebra on a Hilbert space $H$, $\xi_o\in H$,
$\big(\xi_k\big)_{k\geq 1}$ a sequence in $\overline{M\xi_o}\,$, and
$\big(\gamma_k\big)_{k\geq 1}$ a sequence in $(0\, ,+\infty )$ such that
\begin{equation*}
\sum\limits_{k=1}^{\infty} \frac 1{\;\!\gamma_k\;\!}\;\! \|\xi_k\|^2<+\infty\, .
\end{equation*}
Then there exist $a\in M$ with $0\leq a\leq 1\,$, $\eta_o\in
(\;\! \overline{aH}\;\! )\cap (\;\!\overline{M\xi_o}\;\!) =
\overline{aM\xi_o}\,$, as well as a sequence $\big(b_k\big)_{k\geq 1}$
in the operator norm closure of $Ma$, such that
\begin{equation*}
\begin{split}
&\; a\;\!\eta_o =\xi_o\, , \\
&b_k\eta_o = \xi_k\!\text{ and }\| b_k\|\leq\sqrt{\gamma_k}\, ,\qquad k\geq 1\, .
\end{split}
\end{equation*}
\end{theorem}

\proof
Let $k\geq 1$ be arbitrary. Since $\xi_k\in\overline{M\xi_o}\,$,  we can find by induction
a sequence $\big( x_{\substack{ {} \\ k,j}}\big)_{j\geq 0}$ in $M$ such that
\begin{equation*}
\bigg\|\;\! \xi_k -\sum\limits_{j=0}^n x_{\substack{ {} \\ k,j}}\xi_o\;\!\bigg\| \leq
\frac 1{\;\! 4^{n+2}}\;\! \|\xi_k\|\, ,\qquad n\geq 0\, .
\end{equation*}
Then
\begin{equation*}
\begin{split}
\| x_{\substack{ {} \\ k,o}}\xi_o \| =\;&
\big\|\;\! \xi_k -\big(\xi_k - x_{\substack{ {} \\ k,o}}\xi_o\big)\big\|
\leq \frac{17}{\;\! 16\;\!}\;\! \|\xi_k\|\, ,\\
\| x_{\substack{ {} \\ k,n}}\xi_o \| =\;&
\bigg\|\bigg( \xi_k -\sum\limits_{j=0}^{n-1} x_{\substack{ {} \\ k,j}}\xi_o\bigg)
-\bigg( \xi_k -\sum\limits_{j=0}^n x_{\substack{ {} \\ k,j}}\xi_o\bigg)\bigg\|
\leq \frac 5{\;\! 4^{n+2}}\;\!\|\xi_k\, ,\quad n\geq 1\, .
\end{split}
\end{equation*}
In particular, we can write
\begin{equation}\label{decomposition}
\xi_k =\sum\limits_{j=0}^{\infty} x_{\substack{ {} \\ k,j}}\xi_o
\end{equation}
where the series converges in norm.

Now let us define
\begin{equation*}
y_p=\bigg( 1+\sum\limits_{k=1}^p \sum\limits_{j=0}^p\frac{\;\! 4^{j+1}}{\gamma_k}\;\!
x_{\substack{ {} \\ k,j}}^{\, *}x_{\substack{ {} \\ k,j}}\bigg)^{\! 1/2}\in M\, ,\qquad p\geq 1\, .
\end{equation*}
Since the square root function $\sqrt{t}$ is operator increasing on $[\;\! 0\, ,+\infty )$
(see e.g. \cite{P}, Proposition 1.3.8 or \cite{SZOA}, Proposition 2.7), we have
\medskip

\centerline{$1\leq y_1\leq y_2\leq\, ...\, ,$}
\smallskip

\noindent and since $1/t$ is operator decreasing on $(\;\! 0\, ,+\infty )$
(see e.g. \cite{P}, Proposition 1.3.6 or \cite{SZOA}, 2.6 (7)), we have
\smallskip

\centerline{$1\geq y_1^{\;\! -1}\geq y_2^{\;\! -1}\geq\, ... \,\geq 0\, .$}
\medskip

\noindent Therefore, the sequence $\big( y_p^{\;\! -1}\big)_{p\geq 1}$ is convergent
in the strong operator topology to some $a\in M, 0\leq a\leq 1\, .$

On the other hand, we have for every $p\geq 1$
\medskip

\noindent\hspace{1.26 cm}$\displaystyle \| y_p\xi_o\|^2= (y_p^{\;\! 2}\xi_o\;\! |\;\!\xi_o ) =
\bigg( \xi_o +\sum\limits_{k=1}^p \sum\limits_{j=0}^p\frac{\;\! 4^{j+1}}{\gamma_k}\;\!
x_{\substack{ {} \\ k,j}}^{\, *}x_{\substack{ {} \\ k,j}}\xi_o\,\bigg|\, \xi_o\bigg)$

\noindent\hspace{2.49 cm}$\displaystyle =\|\xi_o\|^2 +\sum\limits_{k=0}^p
\sum\limits_{j=0}^p\frac{\;\! 4^{j+1}}{\gamma_k}\;\!\| x_{\substack{ {} \\ k,j}}\xi_o\|^2$

\noindent\hspace{2.49 cm}$\displaystyle =|\xi_o\|^2 +\sum\limits_{k=0}^p
\frac 4{\;\!\gamma_k\;\!}\;\! \| x_{\substack{ {} \\ k,o}}\xi_o \|^2 +\sum\limits_{k=1}^p
\sum\limits_{j=1}^p\frac{\;\! 4^{j+1}}{\gamma_k}\;\! \| x_{\substack{ {} \\ k,j}}\xi_o\|^2$

\noindent\hspace{2.49 cm}$\displaystyle \leq\|\xi_o\|^2 +\sum\limits_{k=0}^p
\frac 4{\;\!\gamma_k\;\!}\cdot\frac{\;\! 17^2}{\;\! 16^2}\;\! \| \xi_k\|^2 +\sum\limits_{k=1}^p
\sum\limits_{j=1}^p\frac{\;\! 4^{j+1}}{\gamma_k}\cdot\frac{25}{\;\! 4^{2j+4}}\;\!\|\xi_k\|^2$

\noindent\hspace{2.49 cm}$\displaystyle =\|\xi_o\|^2 +\sum\limits_{k=0}^p
\frac 1{\;\!\gamma_k\;\!}\;\! \|\xi_k\|^2\bigg(
\frac{\;\! 289\;\!}{64} +\sum\limits_{j=1}^p\frac{25}{\;\! 4^{j+3}}\bigg)$

\noindent\hspace{2.49 cm}$\displaystyle \leq\|\xi_o\|^2 +\sum\limits_{k=0}^p
\frac 1{\;\!\gamma_k\;\!}\;\! \|\xi_k\|^2\bigg(
\frac{\;\! 289\;\!}{64} +\sum\limits_{j=1}^{\infty}\frac{25}{\;\! 4^{j+3}}\bigg)$

\noindent\hspace{2.49 cm}$\displaystyle =\|\xi_o\|^2 +\frac{\;\! 223\;\!}{48}
\sum\limits_{k=0}^p\frac 1{\;\!\gamma_k\;\!}\;\! \|\xi_k\|^2 \leq \|\xi_o\|^2 +5\;\!
\sum\limits_{k=0}^{\infty}\frac 1{\;\!\gamma_k\;\!}\;\! \|\xi_k\|^2\, ,$
\medskip

\noindent that is
\begin{equation*}
\| y_p\xi_o\|\leq c\, ,\qquad p\geq 1
\end{equation*}
where
\medskip

\centerline{$\displaystyle c=\bigg(\|\xi_o\|^2 +5\;\! \sum\limits_{k=0}^{\infty}
\frac 1{\;\!\gamma_k\;\!}\;\! \|\xi_k\|^2\bigg)^{\! 1/2}<+\infty\, .$}
\smallskip

\noindent Since closed balls in $H$ are weakly compact, there exists a weakly
convergent subnet $\big( y_{p_\iota}\xi_o\big)_{\iota}$ of the bounded
sequence $\big( y_p\xi_o\big)_{p\geq 1}$ (actually there exists a weakly
convergent subsequence because the closed balls in the closed linear
span of the sequence $\big( y_p\xi_o\big)_{p\geq 1}$ are compact and metrizable).
Let $\eta$ denote the weak limit of $\big( y_{p_\iota}\xi_o\big)_{\iota}\,$.
Clearly, $\eta\in\overline{M\xi_o}\,$.

We claim that $a\;\!\eta =\xi_o\,$. Indeed, for every $\xi\in H$ we have
\begin{equation*}
\begin{split}
\big|\big( y_{p_\iota}\xi_o\;\! \big|\;\! a\;\!\xi \big) -\big(\xi_o\;\! \big|\;\! \xi \big)\big|
=\;&\big|\big( y_{p_\iota}\xi_o\;\! \big|\;\! a\;\!\xi \big) -
\big( y_{p_\iota}\xi_o\;\! \big|\;\! y_{p_\iota}^{\;\! -1}\xi \big)\big| =
\big|\big( y_{p_\iota}\xi_o\;\! \big|\;\! a\;\!\xi -y_{p_\iota}^{\;\! -1}\xi \big)\big| \\
\leq\;&\big\|y_{p_\iota}\xi_o\big\|\cdot\big\| a\;\!\xi -y_{p_\iota}^{\;\! -1}\xi \big\|
\leq c\;\! \big\| a\;\!\xi -y_{p_\iota}^{\;\! -1}\xi \big\|
\end{split}
\end{equation*}
and $y_p^{\;\! -1}\overset{so}{\longrightarrow} a$ yields $\lim\limits_{\iota}
\big( y_{p_\iota}\xi_o\;\! \big|\;\! a\;\!\xi \big) =\big(\xi_o\;\! \big|\;\! \xi \big)\,$.
Taking now into account that the weak limit of $\big( y_{p_\iota}\xi_o\big)_{\iota}$
is $\eta\,$, we conclude that $\big( \eta\;\! \big|\;\! a\;\!\xi \big) =
\big(\xi_o\;\! \big|\;\! \xi \big)\,$, hence
\smallskip

\centerline{$\big( a\;\!\eta\;\! \big|\;\! \xi \big) =\big( \eta\;\! \big|\;\! a\;\!\xi \big) =
\big(\xi_o\;\! \big|\;\! \xi \big)\, .$}
\smallskip

Let $\eta_o$ denote the orthogonal projection of $\eta$ onto $\overline{aH}$,
that is $\eta_o= s(a)\eta$ where $s(a)\in M$ is the support projection of $a\,$.
Then $a\;\!\eta_o =a\;\! s(a)\eta =a\;\!\eta =\xi_o\,$.

Denoting now by $p_{\xi_o}{\!\!\!\! '}$ the cyclic projection in $M'$ associated
to $\xi_o\,$, that is the orthogonal projection onto $\overline{M\xi_o}\,$, we have
$\eta =p_{\xi_o}{\!\!\!\! '}\;\eta$ and consequently
\smallskip

\centerline{$\eta_o =s(a)\eta =s(a)p_{\xi_o}{\!\!\!\! '}\;\eta =
p_{\xi_o}{\!\!\!\! '}\; s(a)\eta\in\overline{M\xi_o}\, .$}
\smallskip

\noindent Therefore $\eta_o\in (\;\! \overline{aH}\;\! )\cap (\;\!\overline{M\xi_o}\;\!)\,$.

The inclusion $(\;\! \overline{aH}\;\! )\cap (\;\!\overline{M\xi_o}\;\!)\supset
\overline{aM\xi_o}$ is obvious. For the proof of the converse inclusion let $\zeta\in
(\;\! \overline{aH}\;\! )\cap (\;\!\overline{M\xi_o}\;\!)$ be arbitrary. Choosing
a sequence $\big(\eta_k\big)_{k\geq 1}$ in $H$ such that $\zeta =
\lim\limits_{k}a\;\!\eta_k\,$, we have

\centerline{$\zeta =p_{\xi_o}{\!\!\!\! '}\;\zeta =\lim\limits_{k}p_{\xi_o}{\!\!\!\! '}\;
a\;\!\eta_k =\lim\limits_{k}a\;\! p_{\xi_o}{\!\!\!\! '}\;\eta_k\, .$}

\noindent Since $p_{\xi_o}{\!\!\!\! '}\;\eta_k\in \overline{M\xi_o}\, ,k\geq 1\,$,
it follows that $\zeta$ belongs to the closure of $a\overline{M\xi_o}\,$, that
is to $\overline{aM\xi_o}\,$.

Summing up the above, we have $a\in M\, ,\;\! 0\leq a\leq 1\,$, and $\eta_o\in
(\;\! \overline{aH}\;\! )\cap (\;\!\overline{M\xi_o}\;\!) =\overline{aM\xi_o}$ such
that
\begin{equation}\label{eta}
a\;\!\eta_o =\xi_o\, .
\end{equation}

Let now $k\geq 1$ and $j\geq 0$ be arbitrary. For every $p\geq \max\big( k\, ,j\big)$
we have
\smallskip

\noindent\hspace{1.37 cm}$\displaystyle
y_p^{\;\! -1}\bigg(\frac{\;\! 4^{j+1}}{\gamma_k}\;\! x_{\substack{ {} \\ k,j}}^{\, *}
x_{\substack{ {} \\ k,j}}\bigg) y_p^{\;\! -1}
\leq y_p^{\;\! -1}\bigg( \sum\limits_{k'=1}^p \sum\limits_{j'=0}^p
\frac{\;\! 4^{j+1}}{\gamma_{k'}}\;\!
x_{\substack{ {} \\ k',j'}}^{\, *}x_{\substack{ {} \\ k',j'}}\bigg)^{\! 1/2} y_p^{\;\! -1}$

\noindent\hspace{5.125 cm}$\displaystyle
=y_p^{\;\! -1}\big( y_p^{\;\! 2}-1\big) y_p^{\;\! -1} =1-y_p^{\;\! -2}$
\medskip\smallskip

\noindent\hspace{5.125 cm}$\leq 1\, .$
\smallskip

\noindent Taking into account that $y_p^{\;\! -1}\overset{so}{\longrightarrow} a\,$,
we obtain
\smallskip

\centerline{$\displaystyle a\bigg(\frac{\;\! 4^{j+1}}{\gamma_k}\;\! x_{\substack{ {} \\ k,j}}^{\, *}
x_{\substack{ {} \\ k,j}}\bigg) a\leq 1\,\Longleftrightarrow\,
a\;\! x_{\substack{ {} \\ k,j}}^{\, *}x_{\substack{ {} \\ k,j}}a\leq\frac{\gamma_k}{\;\! 4^{j+1}}
\,\Longleftrightarrow\, \| x_{\substack{ {} \\ k,j}}a\|\leq\frac{\sqrt{\gamma_k}}{\;\! 2^{j+1}}\, .$}
\medskip

By the above estimation, we can define
\begin{equation*}
b_k =\sum\limits_{j=0}^{\infty} x_{\substack{ {} \\ k,j}}a\, ,\qquad k\geq 1
\end{equation*}
where the series converges in the norm and thus $b_k$ belongs to the operator norm closure
of $Ma\,$. Moreover,
\begin{equation*}
\| b_k\|\leq \sum\limits_{j=0}^{\infty} \| x_{\substack{ {} \\ k,j}}a\| \leq
\sum\limits_{j=0}^{\infty} \frac{\sqrt{\gamma_k}}{\;\! 2^{j+1}} \leq \sqrt{\gamma_k}\, ,
\qquad k\geq 1\, .
\end{equation*}
By (\ref{eta}) and (\ref{decomposition}) holds also
\smallskip

\centerline{$\displaystyle b_k\eta_o=\sum\limits_{j=0}^{\infty} x_{\substack{ {} \\ k,j}}a\;\!\eta_o
=\sum\limits_{j=0}^{\infty} x_{\substack{ {} \\ k,j}}\xi_o =\xi_k\, ,\qquad k\geq 1\, .$}

\hfill $\square\quad$
\medskip

Now we show how the above theorem can be used to "lift" vector sequences which
converge sufficiently fastly to zero in operator sequences which converge to zero
in operator norm.

\begin{corollary}\label{BT-convergence}
{\rm (BT-Theorem for convergence)}
Let $M$ be a von Neumann algebra on a Hilbert space $H$, $\xi_o\in H$, and
$\big(\xi_k\big)_{k\geq 1}$ a sequence in $\overline{M\xi_o}$ such that
\smallskip

\centerline{$\displaystyle \sum\limits_{k=1}^{\infty} \|\xi_k\|^2<+\infty\, .$}
\smallskip

\noindent Then there exist $a\in M$ with $0\leq a\leq 1\,$, $\eta_o\in
(\;\! \overline{aH}\;\! )\cap (\;\!\overline{M\xi_o}\;\!) =
\overline{aM\xi_o}\,$, as well as a sequence $\big(b_k\big)_{k\geq 1}$
in the operator norm closure of $Ma$ satisfying
\smallskip

\noindent\hspace{4.9 cm}$a\;\!\eta_o =\xi_o\, ,$
\smallskip

\noindent\hspace{4.8 cm}$b_k\eta_o = \xi_k\!\text{ for }k\geq 1\, ,$
\smallskip

\noindent\hspace{4.8 cm}$\lim\limits_{k\to\infty}\| b_k\| =0\, .$
\end{corollary}

\proof
It is well known that for any convergent series $\displaystyle \sum\limits_{k=0}^{\infty}
\alpha_k$ of positive numbers there is a sequence $0<\gamma_k\longrightarrow 0$
such that the series $\displaystyle \sum\limits_{k=0}^{\infty}\frac 1{\;\!\gamma_k\;\!}\;\!
\alpha_k$ is still convergent (see e.g. \cite{K}, \S 39, 175.4 or \cite{C-Z2}, Lemma 1.5):
we can take, for example,
\bigskip

\centerline{$\displaystyle \gamma_k=\begin{cases}
\bigg(\sum\limits_{j=k}^{\infty}\alpha_j\bigg)^{\! 1/2}+
\bigg(\sum\limits_{j=k+1}^{\infty}\alpha_j\bigg)^{\! 1/2} & \text{if }\alpha_k>0\, , \\
\hspace{1.9 cm}\displaystyle \frac 1{\;\! 2^k} & \text{if }\alpha_k=0\, .
\end{cases}$}

Applying the above remark to the series $\displaystyle \sum\limits_{k=0}^{\infty}
\|\xi_k\|^2$ we get a sequence $0<\gamma_k\longrightarrow 0$ such that
\smallskip

\centerline{$\displaystyle \sum\limits_{k=1}^{\infty} \frac 1{\;\!\gamma_k\;\!}\;\! \|\xi_k\|^2
<+\infty\, .$}
\smallskip

\noindent Now we can apply Theorem \ref{BT} obtaining $a\, ,\eta_o$ and the
sequence $\big(b_k\big)_{k\geq 1}$ having the desired properties.

\hfill $\square\quad$
\medskip

We note that the classical BT-Theorem follows with $\xi_k=0$, identically for $k\geq 2\,$.

\medskip
\section{Automatic continuity properties of standard von Neumann algebras}
\bigskip

Let $M$ be a standard von Neumann algebra on a Hilbert space $H$.
Corollary \ref{Griffin-standard} claims the boundedness
of every linear operator on $H$ which commutes with all operators belonging to $M$,
that is which intertwines the identity map on $M$ with itself. Using Corollary
\ref{BT-convergence} we shall next prove a general continuity theorem, which implies
the continuity of additive maps intertwining the identity map on $M$ with an
arbitrary bounded real linear map from $M$ into the bounded linear operators on
some Banach space $X\,$.

We recall (see e.g. \cite{Wil}, Problem 5-3-103) : if $X\, ,Y$ are topological vector
spaces and $T : X\longrightarrow Y$ is a $\mathbb{Q}$-homogeneous mapping
which has closed graph, then $T$ is $\mathbb{R}$-homogeneous. Indeed, for any
$x\in X$ and real $\lambda\,$, choosing a sequence $\big(\lambda_k\big)_{k\geq 1}$
of rational numbers converging to $\lambda\,$, we have
\medskip

\centerline{$\big(\lambda_k x\, ,T(\lambda_k x)\big) =\big(\lambda_k x\, ,
\lambda_kT(x)\big)\longrightarrow \big(\lambda x\, ,\lambda T(x)\big) ,$}
\smallskip

\noindent so $\big(\lambda x\, ,\lambda T(x)\big)$ belongs to the graph of $T$.

Consequently, if $X\, ,Y$ are Banach spaces and $T : X\longrightarrow Y$ is an
additive, hence $\mathbb{Q}$-linear map, having closed graph, then $T$ is real
linear and by the closed graph theorem it follows also its boundedness.

\begin{theorem}\label{commutant}
{\rm (Boundedness of intertwining operators)}
Let $M$ be a von Neumann algebra on a Hilbert space $H$. Then
{\rm (i)}$\,\Longrightarrow${\rm (ii)}$\,\Longrightarrow${\rm (iii)} where $:$
\begin{itemize}
\item[(i)] The weak ${}^*$topology on $M'$ coincides with the weak operator topology,
that is every normal positive linear functional on $M'$ is a finite sum of functionals of the
form $\omega_\zeta{\!\! '} : M'\ni x' \longmapsto (x'\zeta |\zeta )\, ,\;\! \zeta\in H$.
\item[(ii)] For every sequence $\big(\xi_k\big)_{k\geq 1}$ in $H$, there exist $n\geq 1$
and $\zeta_1\, ,\, ...\, ,\zeta_n\in H$ such that $\xi_k$ belongs to the closure of
$\;\displaystyle \sum\limits_{j=1}^n M\zeta_j$ for each $k\geq 1\,$.
\item[(iii)] For any Banach space $X$, bounded real linear map
$\Phi$ of $M$ into the Banach space $B(X)$ of all bounded linear operators on $X$, and
additive operators $T_1, T_2 : H\longrightarrow X\!$ satisfying the intertwining condition
\begin{equation}\label{inter}
T_1x=\Phi(x)\;\!T_2\, ,\qquad x\in M,
\end{equation}
the operator $T_1$ and the composition of $T_2$ with the canonical map of $X\!$
onto the quotient Banach space $X\big/\!\! \bigcap\limits_{x\in M}\!\!{\rm Ker}\;\!\Phi (x)$
$($identifiable with 
$\bigvee\limits_{x\in M}r\big(\Phi(x)\big)T_2$ if $X$ is a Hilbert space$\;\! )\!$
are both necessarily real linear and bounded.
\end{itemize}
\end{theorem}

\proof
To show (i)$\;\!\Longrightarrow$(ii) let us assume that (i) holds and let
$\big(\xi_k\big)_{k\geq 1}$ be a sequence in $H$.

The support projection $s(\omega_{\xi_k}{\!\! '})$ of $\omega_{\xi_k}{\!\! '} : M'\ni x'
 \longmapsto (x'\xi_k |\xi_k )$ is the orthogonal
projection onto $\overline{M\xi_k}\,$, so the range of the support projection $s(\varphi')$ of
the normal positive linear functional
\medskip

\centerline{$\displaystyle \varphi' =\sum\limits_{k=1}^{\infty}
\frac 1{\;\! 2^k(1+\|\xi_k\|^2)\;\!}\,\omega_{\xi_k}{\!\! '}$}

\noindent on $M'$, which is $\displaystyle \bigvee\limits_{k=1}^{\infty}\!
s(\omega_{\xi_k}{\!\! '})\,$, contains the sequence $\big(\xi_k\big)_{k\geq 1}\,$.
But by (i) there exist finitely many $\zeta_1\, ,\, ...\, ,\zeta_n\in H$ such that
$\displaystyle \varphi' =\sum\limits_{j=1}^n\omega_{\zeta_j}{\!\! '}$ and so
$s(\varphi') = \displaystyle \bigvee\limits_{j=1}^n\! s(\omega_{\zeta_j}{\!\! '})$
is the orthogonal projection onto the closure of $\;\displaystyle \sum\limits_{j=1}^n
M\zeta_j\,$.

For (ii)$\;\!\Longrightarrow$(iii) let us assume that (ii) holds and let $X$ be a Banach
space, $\Phi: M\longrightarrow B(X)$ a bounded real linear map, and
$T_1\, ,T_2 : H\longrightarrow X$ additive maps satisfying the intertwining condition
\eqref{inter}.

To prove the continuity of $T_1$ it is enough to verify its continuity in $0\,$,
which follows once we prove  that for every sequence $\big(\xi_k\big)_{k\geq 1}$ in $H$
with
\smallskip

\centerline{$\displaystyle \|\xi_k\|\leq\frac 1{\;\! 2^k}\, ,\qquad k\geq 1$}
\smallskip

\noindent the convergence $T_1\xi_k\longrightarrow 0$ holds true.

By (ii) there exist finitely many $\zeta_1\, ,\, ...\, ,\zeta_n\in H$ such that each $\xi_k$
belongs to the closure of $\;\displaystyle \sum\limits_{j=1}^n M\zeta_j\,$.

Let $M_n(M)$ denote the von Neumann algebra of all $n\times n$ matrices with
entries in $M$, acting on the Hilbert space $\displaystyle H_n =
\bigoplus\limits_{j=1}^n H$. With
\smallskip

\centerline{$\displaystyle \widetilde{\zeta} =\bigoplus\limits_{j=1}^n \zeta_j\in H_n\, ,
\quad \widetilde{\xi}_k =\xi_k\oplus \underbrace{0\oplus\, ...\,\oplus 0}_{n-1\,\text{times}}
\in H_n\, , k\geq 1$}
\smallskip

\noindent we have $\displaystyle \sum\limits_{k=1}^{\infty}\|\widetilde{\xi}_k\|^2 =
\sum\limits_{k=1}^{\infty}\| \xi_k\|^2<+\infty$ and $\widetilde{\xi}_k\in
\overline{M_n(M)\;\!\widetilde{\zeta}}\, ,k\geq 1\,$, so we can apply Corollary
\ref{BT-convergence} obtaining (among other things) a vector $\widetilde{\eta}\in H_n$
and a sequence $\big(\;\!\widetilde{b}_k\big)_{k\geq 1}$ in $M_n(M)$ with
\smallskip

\noindent\hspace{4.72 cm}$\widetilde{b}_k\widetilde{\eta} =
\widetilde{\xi}_k\, ,\qquad k\geq 1\, ,$
\smallskip

\noindent\hspace{4.72 cm}$\lim\limits_{k\to\infty}\| \widetilde{b}_k\| =0\, .$
\smallskip

\noindent If $b_{k1}\, ,\, ...\, , b_{kn}$ is the first row of the matrix $\widetilde{b}_k$
and $\eta_1\, ,\, ...\, ,\,\eta_n$ are the components of $\widetilde{\eta}$ then
\smallskip

\noindent\hspace{3.8 cm}$\displaystyle \xi_k =\sum\limits_{j=1}^n b_{kj}\;
\!\eta_j\, ,\qquad k\geq 1\, ,$

\centerline{$\| b_{kj}\|\leq \| \widetilde{b}_k\|\, ,\qquad k\geq 1\text{ and }1\leq j\leq n\, .$}
\smallskip

\noindent Consequently
\begin{equation*}
\begin{split}
\| T_1\xi_k\| =\;&\bigg\|\sum\limits_{j=1}^n T_1\;\! b_{kj}\;\!\eta_j\bigg\|
=\bigg\|\sum\limits_{j=1}^n\Phi(b_{kj})T_2\;\!\eta_j\bigg\|\leq\|\Phi\|\sum\limits_{j=1}^n
\| b_{kj}\|\!\cdot\!\| T_2\eta_j\| \\
\leq\;& \|\Phi\|\| \widetilde{b}_k\|\sum\limits_{j=1}^n\| T_2\eta_j\| \longrightarrow 0
\end{split}
\end{equation*}
and so $T_1\xi_k\longrightarrow 0\,$.
\smallskip

Having verified the continuity of the additive, hence $\mathbb{Q}$-linear map
$T_1\,$, its real linearity follows immediately.

We go next to prove that the graph of the composition $\widehat{T_2}$ of
$T_2$ with the canonical map of $X\!$ onto the quotient Banach space
$X\big/\!\! \bigcap\limits_{x\in M}\!\!{\rm Ker}\;\!\Phi (x)$ is closed.

We shall denote the canonical image of $\zeta\in X$ by $\widehat{\zeta}\,$,
so that $\widehat{T_2}(\xi ) =\widehat{T_2(\xi )}\;\! ,\xi\in H$.
Taking into account the additivity of $T_2\,$, it is enough to prove that if
$\big(\xi_k\big)_{k\geq 1}\subset H$ is a sequence such that $\xi_k\longrightarrow 0$
and $\widehat{T_2\xi_k}\longrightarrow \widehat{\zeta}\,$, then $\widehat{\zeta} =
\widehat{0}\,$, that is
\begin{equation}\label{kernel}
\Phi (x)\zeta =0\, ,\qquad x\in M\, .
\end{equation}

For let $x\in M$ be arbitrary. If $\varphi$ is any bounded linear functional on $X$,
then the composition $\varphi\circ\Phi (x)$ is a bounded linear functional on $X$
which vanishes on $\bigcap\limits_{y\in M}\!\!{\rm Ker}\;\!\Phi (y)\;\!$, defining thus
the bounded linear functional
\smallskip

\centerline{$\widehat{\varphi\circ\Phi (x)} :
X\big/\!\! \bigcap\limits_{x\in M}\!\!{\rm Ker}\;\!\Phi (x)\ni\widehat{\eta}
\longmapsto \varphi \big(\Phi (x)\eta\big) .$}
\smallskip

\noindent Since by \eqref{inter} and by the boundedness of $T_1$, we have
\[
\begin{split}
\varphi \big(\Phi (x)\zeta\big) =\;&\big( \widehat{\varphi\circ\Phi (x)}\big)
(\widehat{\zeta})
= \lim\limits_{k\to\infty}\big( \widehat{\varphi\circ\Phi (x)}\big)
\big(\widehat{T_2\xi_k}\big) \\
=\;&\lim\limits_{k\to\infty} \varphi\big( \Phi (x)\;\! T_2\xi_k\big)
= \lim\limits_{k\to\infty} \varphi\big( T_1x\xi_k\big) \\
=\;&0
\end{split}
\]
for any $\varphi\,$, the Hahn-Banach theorem yields (\ref{kernel}).

We conclude that the additive, hence $\mathbb{Q}$-linear map $\widehat{T_2}$
has closed graph and, by the remarks before the statement of the theorem, it follows
its real-linearity and continuity.

\hfill $\square\quad$
\medskip

In particular :

\begin{corollary}\label{commutant-by-support}
Let $M$ be a von Neumann algebra on a Hilbert space $H$ such that every
normal positive linear functional on $M'$ is a vector functional, that is of the form
$\omega_\xi{\!\! '} : M'\ni x'\longmapsto (x'\xi |\xi )$ for an appropriate  $\xi\in H$.
Then statement ${\rm (iii)}$ in Theorem \ref{commutant} holds true.



\end{corollary}

\hfill $\square\quad$
\medskip

\begin{corollary}\label{commutant-cases}
{\rm (Cases of automatic boundedness of intertwining operators)}
Let $M$ be a von Neumann algebra on a Hilbert space $H$.
Then statement ${\rm (iii)}$ in Theorem \ref{commutant} holds true
in each one of the following situations $:$


\begin{itemize}
\item[(1)] $M$ has a cyclic vector;
\item[(2)] $M$ is properly infinite;
\item[(3)] $M$ is standard.
\end{itemize}
\end{corollary}

\proof
(1) is an immediate consequence of Theorem \ref{commutant}.

(2) and (3) follow from Corollary \ref{commutant-by-support} because every normal
positive linear form on the commutant of a properly infinite von Neumann algebra
is of the form $\omega_\xi{\!\! '}$ for some $\xi\in H$ (see e.g. \cite{S-Z}, Theorem 8.16)
and standard von Neumann algebras have the same property (see \cite{A},
Theorem 6, \cite{Co2}, Th\'eor\`eme 2.7, \cite{Haa1}, Lemma 2.10).

\hfill $\square\quad$
\medskip

\begin{remark}
With $X=H$, $\Phi={\rm id}$ and $T_1=T_2$, statement {\rm (iii)} in Theorem
\ref{commutant} reduces to the automatic continuity of any additive map
$T :  H\longrightarrow H$ satisfying
$$
Tx=x\;\!T\, ,\qquad x\in M\, .
$$
Therefore Corollary \ref{commutant-cases} can be used to obtain
an alternative proof of Corollary \ref{Griffin-standard}.
We point out the fact that the more general situation in Theorem \ref{commutant}
may not be treated by the method of \cite{Gr2}.

\end{remark}

\medskip
\section{Some applications}
\label{applx}
\bigskip

The present section is devoted to provide some applications of the previous result (cf. Theorem
\ref{commutant}) on the boundedness of intertwining operators. 

We first start with the algebra homomorphism $x\mapsto TxT^{-1}$ implemented in a canonical
way by an invertible, a priori non necessarily bounded, linear operator.

\begin{proposition}\label{Implementation}
Let $M$ be a von Neumann algebra on a Hilbert space $H$, and $T$ a bijective
linear map from $H$ onto a Banach space $X$ such that $\big\{ TxT^{-1}\,; x\in M\big\}$
is a closed subalgebra of the Banach algebra $B(X)$ of all bounded linear operators
on $X$. Assuming that $M$ satisfies one of the conditions of Corollary \ref{commutant-cases},
the operator $T$ must be necessarily bounded.
\end{proposition}

\begin{proof}
Let us denote by $\Phi$ the injective algebra homomorphism
\smallskip

\centerline{$M\ni x\longmapsto TxT^{-1}\in B(X)\, .$}
\smallskip

\noindent By our assumption $\Phi (M)$ is a closed, hence complete subalgebra of
$B(X)$, and $\Phi$ is an algebra isomorphism of $M$ onto it. Since the von Neumann
algebra $M$ is semi-simple and semi-simplicity is an algebraic invariant (see e.g.
\cite{Bon-Du}, Chapter III, $\S$ 24, Definition 13), $\Phi (M)$ is a semi-simple algebra
endowed with the two complete norms:
\smallskip

\centerline{$\Phi (x)\longmapsto \|\Phi (x)\|$ and $\Phi (x)\longmapsto \| x\|\, .$}
\smallskip

\noindent By a classical theorem of B. E. Johnson (\cite{Jo}, see also \cite{Bon-Du},
Chapter III, $\S$ 25, Theorem 9, a short proof was done in \cite{Au1}), these norms
should be equivalent and therefore the map $\Phi : M\longrightarrow B(X)$ is bounded.

Applying now Corollary \ref{commutant-cases}, we end the proof.
\end{proof}

Now let $\varphi$ be a n.s.f. weight on a $W^*$-algebra $M$. We consider the left ideal
\smallskip

\centerline{${\mathfrak N}_\varphi =\{ x\in M\, ;\, \varphi (x^*x)<+\infty\}$}
\smallskip

\noindent and denote by $\pi_{\varphi} : M\longrightarrow B(H_{\varphi})$ the associated
GNS representation. $x_\varphi$ will stay for $x\in{\mathfrak N}_\varphi$ considered an
element of the Hilbert space $H_\varphi$ and we shall discuss the existence of an
invariant algebraic complement of ${\mathfrak N}_\varphi$ in $H_\varphi\,$.

For $\tau$ a n.s.f. trace and $A$ a (possibly unbounded) positive, self-adjoint linear
operator in $H_\tau\,$, affiliated with $\pi_\tau (M)\,$, we shall use the notation of
Pedersen-Takesaki $\tau (A\,\cdot\, )$ for the normal, semi-finite weight defined by
\[
\begin{split}
\tau (A\;\! b):=\;&\lim\limits_{k\to\infty}\tau\Big( \pi_\tau^{-1}\big( A^{1/2}
\chi_{\substack{ {} \\ [0,k]}}(A)\big)\cdot b\cdot\pi_\tau^{-1}\big( A^{1/2}
\chi_{\substack{ {} \\ [0,k]}}(A)\big)\Big) \\
=\;&\lim\limits_{k\to\infty}\tau\Big( b^{1/2}\pi_\tau^{-1}\big( A\;\!\chi_{\substack{ {} \\ [0,k]}}(A)\big)
b^{1/2}\Big)\, ,\qquad 0\leq b\in M\, ,
\end{split}
\]
where $\displaystyle \chi_{\substack{ {} \\ [0,k]}}$ stands for the
characteristic function of $[0,k]$ (see \cite{P-T}, Paragraph 4).
The weight $\tau (A\,\cdot\, )$ is faithful if and only if $A$ is injective
and $\tau (A\,\cdot\, )\geq\lambda\;\!\tau$ for some scalar $\lambda >0$
if and only if $A\geq\lambda\,$.

\begin{proposition}\label{complement}
Let $M\neq\{ 0\}$ be a $W^*$-algebra equipped with a n.s.f. weight $\varphi\,$, and
$\mathfrak X$ a left ideal of $M$, contained in ${\mathfrak N}_\varphi$ and such that
$X=\{ x_\varphi\, ;\, x\in\mathfrak X\}$ is dense in $H_\varphi\,$. Then the
following statements are equivalent $:$
\begin{itemize}
\item[(i)] $X$ admits an algebraic complement in $H_\varphi$ which is
invariant under the action of $\pi_\varphi (M)\,$.
\item[(ii)] $\mathfrak X=\mathfrak N_{\varphi}$ and $\{ x_\varphi\, ;\, x\in
\mathfrak N_{\varphi}\} =H_{\varphi}\,$.
\item[(iii)] $\mathfrak X=\mathfrak N_{\varphi}$ and
$\displaystyle \sup\limits_{\substack{ 0\leq b\in M \\ b\neq 0}}\frac{\| b\|}{\;\!\varphi (b)\;\!}
<+\infty\,$.
\item[(iv)] $\mathfrak X=\mathfrak N_{\varphi}$ and the reduced $W^*$-algebra
$eMe$ is finite-dimensional for every projection $e\in M$ with $\varphi (e)<+\infty\,$.
\item[(v)] $\mathfrak X=\mathfrak N_{\varphi}\,$, $M$ is the direct product of a family
$\big( M_\iota \big)_{\iota\in I}$ of type $I$ factors and, denoting by $\tau_\iota$ the
canonical trace on $M_\iota$ $($i.e. the n.s.f. trace which is equal to $1$ in every
minimal projection$)$, for some scalar $\lambda >0$ we have
\begin{equation*}
\varphi (b)\geq\lambda \Big( \bigoplus\limits_{\iota\in I}\tau_\iota\Big) (b)\, ,\qquad
0\leq b\in M\;\! .
\end{equation*}
\end{itemize}
\end{proposition}

\begin{proof}
The implication $(ii)\Longrightarrow (i)$ is trivial.

For $(i)\Longrightarrow (ii)$ let us assume that there exists a $\pi_\varphi (M)$-invariant
linear subspace $Y$ of $H_\varphi$ such that $H_\varphi =X\oplus Y$, where "$\oplus$"
stands for inner algebraic direct sum. Let $P$ denote the corresponding projection
operator: $P : H_\varphi\longrightarrow H_\varphi$ is the linear operator defined by
$P(\xi +\eta )=\xi$ where $\xi\in X,\eta\in Y$. It is easily seen that $P$ satisfies the
commutation condition
\medskip

\centerline{$P\pi_\varphi (x) =\pi_\varphi (x) P\, ,\qquad x\in M$}
\smallskip

\noindent and therefore, by Corollary \ref{Griffin-standard} (or by Corollary \ref{commutant-cases}),
we infer that $P$ is a bounded operator. Consequently, the dense  linear subspace
$X=P H_\varphi$ of $H_\varphi$ is also closed, hence $X=H_\varphi\,$. Now
\smallskip

\centerline{$H_\varphi =X\subset \{ x_\varphi\, ;\, x\in\mathfrak N_{\varphi}\}\subset H_\varphi$}
\smallskip

\noindent means $X=\{ x_\varphi\, ;\, x\in\mathfrak N_{\varphi}\}\Longleftrightarrow \mathfrak X=
\mathfrak N_{\varphi}$ and $\{ x_\varphi\, ;\, x\in\mathfrak N_{\varphi}\} =H_\varphi\,$, that is
(ii).

For the proof of equivalence $(ii)\Longleftrightarrow (iii)$ we shall use the fact that
\begin{equation}\label{closed-graph}
\begin{split}
&\big\{ (\xi ,x)\, ;\, x\in\mathfrak N_{\varphi}\;\! ,\xi =x_\varphi\big\}\text{ is a closed subset of }H_\varphi
\times M \\
&\text{ with respect to the product of the norm-topologies}
\end{split}
\end{equation}
(in other words the linear operator $\mathfrak N_{\varphi}\ni x\longmapsto x_\varphi\in
H_\varphi$ is closed with respect to the norm-topologies; cf. with the Lebesgue continuity
property considered in \cite{Z1}, Section 2 and \cite{A-Z}, Section 2.4).

To verify (\ref{closed-graph}) let us assume that $\xi\in H_\varphi\, ,x\in M$ are such that,
for some sequence $(x_k)_{k\geq 1}$ in $\mathfrak N_{\varphi}\,$, we have
$\|\xi -(x_k)_\varphi\|\longrightarrow 0$ and $\| x-x_k\|\longrightarrow 0\,$. Then by the
weak$^*$ lower semicontinuity of the normal weight $\varphi\,$, we have
\medskip

\centerline{$\displaystyle \varphi (x^*x)\leq \varliminf\limits_{k\to\infty}\varphi (x_k^{\;\! *}x_k) =
\varliminf\limits_{k\to\infty}\| (x_k)_\varphi\|^2 =\|\xi\|^2<+\infty\, .$}
\smallskip

\noindent In particular $x\in\mathfrak N_{\varphi}$, and for every $y\in\mathfrak N_{\varphi}$
we obtain :
\[
\begin{split}
J_\varphi\pi_\varphi (y) J_\varphi x_\varphi =\;&\pi_\varphi (x)J_\varphi y_\varphi =
\lim\limits_{k\to\infty} \pi_\varphi (x_k)J_\varphi y_\varphi \\
=\;&\lim\limits_{k\to\infty} J_\varphi\pi_\varphi (y) J_\varphi (x_k)_\varphi =
J_\varphi\pi_\varphi (y) J_\varphi\;\! \xi\, .
\end{split}
\]
Taking into account that the identity operator on $H_\varphi$ belongs to the weak operator closure of
$J_\varphi\pi_\varphi (\mathfrak N_{\varphi}) J_\varphi\,$, we conclude that $x_\varphi =\xi\,$.

Now $(ii)\Longrightarrow (iii)$ follows by a simple application of the closed graph theorem.
Indeed, if $\{ x_\varphi\, ;\, x\in\mathfrak N_{\varphi}\} =H_{\varphi}$ then we can consider
the everywhere defined linear operator
\smallskip

\centerline{$H_{\varphi}=\{ x_\varphi\, ;\, x\in\mathfrak N_{\varphi}\} \ni x_\varphi
\longmapsto x\in M$}
\smallskip

\noindent whose graph is closed by (\ref{closed-graph}). According to the closed
graph theorem, there exists a constant $c\geq 0$ such that $\| x\|\leq c\;\!\varphi (x^*x)^{1/2}$
holds true for all $x\in\mathfrak N_{\varphi}\,$. Consequently,
\medskip

\centerline{$\displaystyle \frac{\| b\|}{\;\!\varphi (b)\;\!} =
\bigg( \frac{\| b^{1/2}\|}{\;\!\varphi \big( (b^{1/2})^*b^{1/2}\big)^{1/2}\;\!}\bigg)^{\! 2}\leq c^2\, ,
\qquad 0\leq b\in M ,0<\varphi (b)<+\infty\, .$}

\noindent Actually, the inequality $\displaystyle \frac{\| b\|}{\;\!\varphi (b)\;\!} \leq c^2$ holds
for any non-zero $0\leq b\in M$ because for $\varphi (b)=+\infty$ we have
$\displaystyle \frac{\| b\|}{\;\!\varphi (b)\;\!} =\frac{\| b\|}{\;\! +\infty\;\!}=0\leq c^2$.

For the converse implication let us assume that $(iii)$ holds and put
\medskip

\centerline{$\displaystyle c:=\bigg( \sup\limits_{\substack{ 0\leq b\in M \\ b\neq 0}}
\frac{\| b\|}{\;\!\varphi (b)\;\!}\bigg)^{\! 1/2}<+\infty\, .$}
\smallskip

\noindent Then
\smallskip

\centerline{$\displaystyle \| x\| =\| x^*x\|^{1/2}\leq\Big( c^2\varphi (x^*x)\Big)^{\! 1/2}
=c\;\!\| x_\varphi\|\, ,\qquad x\in\mathfrak N_{\varphi}\, .$}
\medskip

\noindent Let now $\xi\in H_\varphi$ be arbitrary and $(x_k)_{k\geq 1}$ a sequence in
$\mathfrak N_{\varphi}$ with $\|\xi -(x_k)_\varphi\|\longrightarrow 0\,$. By the above inequality
$(x_k)_{k\geq 1}$ is a Cauchy sequence with respect to the norm of $M$, so it is
norm-convergent to some $x\in M$. According to (\ref{closed-graph}), we conclude that
$x\in\mathfrak N_{\varphi}$ and $\xi =x_\varphi\,$. This shows that $(ii)$ holds.

Now we have the equivalences $(i)\Longleftrightarrow (ii)\Longleftrightarrow (iii)\,$.
To complete the proof we shall show that $(ii)\Longrightarrow (iv)\Longrightarrow (v)
\Longrightarrow (iii)\,$.

For $(ii)\Longrightarrow (iv)$ let us assume $(ii)$ and let $e\in M$ be any
projection satisfying $\varphi (e)<+\infty\,$.
Since reflexive $C^*$-algebras are finite-dimensional (see the proof of Proposition 2 in
\cite{Sak1} or \cite{KR}, Chapter 10, Exercise 10.5.17 (iii)), it is enough to prove that
every state $\psi$ on $eMe$ is normal.

Let $\widetilde{\psi}$ be the extension of $\psi$ to a state on $M$ defined by
$\widetilde{\psi}(x):=\psi (exe)\,$. Let also denote by $\pi_{\widetilde{\psi}} : M\longrightarrow
B(H_{\widetilde{\psi}})$ the associated GNS representation and by $\xi_{\widetilde{\psi}}$
its canonical cyclic vector, so that $\widetilde{\psi}(x)=\big(\pi_{\widetilde{\psi}}(x)\xi_{\widetilde{\psi}}
\big|\xi_{\widetilde{\psi}}\big)\,$.

The everywhere defined linear operator
\smallskip

\centerline{$T : H_{\varphi}=\{ y_\varphi\, ;\, y\in\mathfrak N_{\varphi}\} \ni y_\varphi
\longmapsto \pi_{\widetilde{\psi}}(y)\xi_{\widetilde{\psi}}\in H_{\widetilde{\psi}}$}
\smallskip

\noindent satisfies the intertwining condition
\smallskip

\centerline{$T \pi_\varphi (x) =\pi_{\widetilde{\psi}}(x)\;\! T\, ,\qquad x\in M .$}
\smallskip

\noindent Indeed, we have for every $x\in M$ and $y\in\mathfrak N_{\varphi}$,
\smallskip

\centerline{$T \pi_\varphi (x)\;\! y_\varphi =T (xy)_\varphi =\pi_{\widetilde{\psi}}(xy)
\xi_{\widetilde{\psi}} =\pi_{\widetilde{\psi}}(x)\big(\pi_{\widetilde{\psi}}(y)
\xi_{\widetilde{\psi}}\big) =\pi_{\widetilde{\psi}}(x)\;\! T\;\! y_\varphi\, .$}
\smallskip

\noindent Applying Corollary \ref{commutant-cases} we infer that $T$ is bounded.

Now we are ready to prove the normalness of $\psi\,$. For let
$\big( b_\kappa\big)_{\kappa\in K}$ be a bounded, increasing net of positive
elements of $eMe$ and $\displaystyle b:=\sup\limits_{\kappa\in K}b_\kappa\,$.
By the normality of $\pi_\varphi$ we have
\smallskip

\centerline{$(b_\kappa)_\varphi =(b_\kappa e)_\varphi =\pi_\varphi (b_\kappa )
e_\varphi \longrightarrow \pi_\varphi (b) e_\varphi =(b\;\! e)_\varphi =b_\varphi$}
\medskip

\noindent and the boundedness of $T$ yields
\medskip

\centerline{$\pi_{\widetilde{\psi}}(b_\kappa)\xi_{\widetilde{\psi}} =T(b_\kappa )_\varphi 
\longrightarrow T\;\! b_\varphi =\pi_{\widetilde{\psi}}(b_\kappa)\xi_{\widetilde{\psi}}\, .$}
\smallskip

\noindent Consequently,
\medskip

\centerline{$\psi (b_\kappa )=\widetilde{\psi}(b_\kappa )=
\big(\pi_{\widetilde{\psi}}(b_\kappa )\xi_{\widetilde{\psi}}\big|\xi_{\widetilde{\psi}}\big)
\longrightarrow \big(\pi_{\widetilde{\psi}}(b)\xi_{\widetilde{\psi}}\big|\xi_{\widetilde{\psi}}\big)
=\widetilde{\psi}(b)=\psi (b)$}
\smallskip

\noindent and we are done.

For $(iv)\Longrightarrow (v)$ let us assume that statement $(iv)$ holds.

First we show that then $M$ is the direct product of a family $\big( M_\iota \big)_{\iota\in I}$
of type $I$ factors. The proof consists in a straightforward application of the Zorn Lemma
once we show that any non-zero central projection $q\in Z(M)$ majorizes some non-zero
central projection $p$ such that $Mp$ is a type $I$ factor.

For we notice that the restriction of $\varphi$ to $Mq$ is semi-finite, hence there exists
a non-zero projection $e_o\in M ,e_o\leq q\,$, having finite weight $\varphi (e_o)<+\infty\,$.
According to $(iv)$ the reduced $W^*$-algebra $e_oMe_o$ is finite-dimensional, so
it contains a minimal projection $0\neq e\leq e_o\leq q\,$.

According to the Zorn Lemma, there exists a maximal set $\mathcal F\supset\{ e\}$ of
mutually orthogonal projections in $M$, all equivalent to $e\,$. Since all projections in
$\mathcal F$ have the same central support $z(e)\leq q\,$,
$\displaystyle \sum\limits_{f\in\mathcal F} f$ is less or equal than $z(e)\,$. We claim that actually
$\displaystyle \sum\limits_{f\in\mathcal F} f =z(e)\,$.
Indeed, by the comparison theorem there exists a central projection $q_o\leq z(e)$
such that
\begin{equation}\label{comparison}
\begin{split}
eq_o\prec\;&\bigg( z(e) -\sum\limits_{f\in\mathcal F} f\bigg) q_o\, , \\
e\big( z(e)-q_o\big)\succ\;&\bigg( z(e) -\sum\limits_{f\in\mathcal F} f\bigg) \big( z(e)-q_o\big)\, .
\end{split}
\end{equation}
If $q_o$ were non-zero, by the minimality of $e$ the projection $0\neq eq_o\leq e$
would be equal to $e$ and the first above relation in (\ref{comparison}) would imply
\medskip

\centerline{$\displaystyle e\prec \bigg( z(e) -\sum\limits_{f\in\mathcal F} f\bigg) q_o
\leq z(e) -\sum\limits_{f\in\mathcal F} f\, ,$}
\smallskip

\noindent contradicting the maximality of $\mathcal F\,$. Thus $q_o=0$ and therefore
the second relation in (\ref{comparison}) yields
\medskip

\centerline{$\displaystyle e\succ z(e) -\sum\limits_{f\in\mathcal F} f\, .$}
\smallskip

\noindent If $\displaystyle z(e) -\sum\limits_{f\in\mathcal F} f$ were $\neq 0$, then
it would be equivalent to a non-zero subprojection of $e\,$, which by the minimality
of $e$ should be equal to $e\,$. But this would contradict the maximality of $\mathcal F\,$.
Consequently $\displaystyle z(e) -\sum\limits_{f\in\mathcal F} f=0\,$.

By the above, $p:=z(e)$ is a non-zero central subprojection of $q\,$, equal to the sum
of the mutually orthogonal, equivalent, minimal projections belonging to $\mathcal F$.
Then $Mp$ is $*$-isomorphic to the von Neumann algebra of all bounded linear
operators on a Hilbert space of dimension ${\rm card}(\mathcal F)$ (see e.g. \cite{S-Z},
Theorem 4.22), and thus is a type $I$ factor.

Now, knowing that $M$ is the direct product of a family $\big( M_\iota \big)_{\iota\in I}$
of type $I$ factors, we can consider on each $M_\iota$ the canonical trace $\tau_\iota\,$,
and then on $M$ the n.s.f. trace $\displaystyle \tau :=\bigoplus\limits_{\iota\in I}\tau_\iota\,$.
By the Radon-Nikodym type theorem of Pedersen and Takesaki (\cite{P-T}, Theorem 5.12)
there exists a positive, self-adjoint linear operator in $H_\tau\,$, affiliated with $\pi_\tau (M)\,$,
such that $\varphi =\tau (A\,\cdot\, )\,$. Using again $(iv)\,$, we shall verify that $0$ does not
belong to the spectrum of $A\,$, that is $A\geq\lambda$ for some scalar $\lambda >0\,$.
Then it will follow $\varphi =\tau (A\,\cdot\, )\geq\lambda\;\!\tau$ and we can conclude that
$(v)$ holds.

To this end let us assume the contrary, that is that $0$ belongs to the spectrum of $A\,$.
Since $\varphi =\tau (A\,\cdot\, )$ is faithful, $A$ must be injective, so $0$ cannot be an
isolated point of the spectrum of $A\,$. Therefore we can find real numbers
\medskip

\centerline{$\alpha_1 >\beta_1 >\alpha_2 >\beta_2 >\, . . .\, >0\, ,\qquad
\alpha_k<2^{-k}\text{ for all }k\geq 1\, ,$}
\smallskip

\noindent such that every interval $(\alpha_k ,\beta_k)$ intersects the spactrum of $A\,$.
Thus the mutually orthogonal spectral projections $\chi_{\substack{ {} \\ (\alpha_k,\beta_k)}}(A)
\in\pi_\tau (M)$ are all non-zero. Choose for every $k\geq 1$ some minimal projection
$f_k\in M$ less or equal than $\pi_\tau^{-1}\big(\chi_{\substack{ {} \\ (\alpha_k,\beta_k)}}(A)\big)$
and put $\displaystyle f_o:=\sum\limits_{k\geq 1}f_k\,$.

Then $\varphi (f_o)<+\infty\,$. Indeed, we have for every $k\geq 1$
\[
\begin{split}
\varphi (f_k) =\;&\varphi\Big( \pi_\tau^{-1}\big(\chi_{\substack{ {} \\ (\alpha_k,\beta_k)}}(A)\big)
f_k\Big) =\tau\Big( f_k\pi_\tau^{-1}\big( A\;\!\chi_{\substack{ {} \\ (\alpha_k,\beta_k)}}(A)\big) f_k\Big) \\
\leq\;&\alpha_k\;\! \tau\Big( f_k\pi_\tau^{-1}\big(\chi_{\substack{ {} \\ (\alpha_k,\beta_k)}}(A)\big) f_k\Big)
=\alpha_k\;\! \tau\big( f_k\big) =\alpha_k \\
<\;& 2^{-k}\, ,
\end{split}
\]
hence, by the normalness of $\varphi\,$,
\bigskip

\centerline{$\displaystyle \varphi (f_o)=\sum\limits_{k\geq 1}\varphi (f_k)\leq\sum\limits_{k\geq 1}
2^{-k} =1<+\infty\, .$}
\smallskip

\noindent By $(iv)$ it should follow that the reduced algebra $f_oMf_o$ is finite-dimensional,
But $f_oMf_o$ contains the infinitely many mutually orthogonal non-zero projections $f_k$
and this contradiction shows that $0$ cannot belong to the spectrum of $A\,$.

Finally, for $(v)\Longrightarrow (iii)$ let us assume that $(v)$ holds.
Since the Hilbert-Schmidt norm majorizes the operator norm, we have
\medskip

\centerline{$\tau_\iota (x_\iota^{\;\! *}x_\iota )^{1/2}\geq \| x_\iota\|\, ,\qquad x_\iota\in
M_\iota\;\! , \iota\in I .$}
\medskip

\noindent Denoting $\displaystyle \tau :=\bigoplus\limits_{\iota\in I}\tau_\iota\,$, it follows
for every $\displaystyle x=\prod\limits_{\iota\in I} x_\iota\in M$
\smallskip

\centerline{$\displaystyle \tau (x^*x) =\sum\limits_{\iota\in I}\tau_\iota (x_\iota^{\;\! *}x_\iota )
\geq\sum\limits_{\iota\in I}\| x_\iota\|^2\geq\sup\limits_{\iota\in I}\| x_\iota\|^2 =\| x\|^2
=\| x^*x\|\, .$}
\smallskip

\noindent Thus we obtain for every non-zero $0\leq b\in M$ with $\varphi (b)<+\infty\,$:
\medskip

\centerline{$\displaystyle \varphi (b)\geq \lambda\;\! \tau (b)\geq\lambda\;\! \| b\|
\Longleftrightarrow \frac{\| b\|}{\;\!\varphi (b)\;\!}\leq \lambda^{-1}\, .$}
\smallskip

\noindent Consequently, $(iii)$ holds true.
\smallskip

\end{proof}
\smallskip

For bounded functionals, Proposition \ref{complement} entails :

\begin{corollary}
Let $M\neq\{ 0\}$ be a $W^*$-algebra equiped with a faithful, normal state $\varphi\,$, and
$\mathfrak X$ a left ideal of $M$ such that $X=\{ x_\varphi\, ;\, x\in\mathfrak X\}$ is dense in
$H_\varphi\,$. Then the following statements are equivalent $:$
\begin{itemize}
\item[(i)] $X$ admits an algebraic complement in $H_\varphi$ which is
invariant under the action of $\pi_\varphi (M)\,$.
\item[(ii)] $\mathfrak X=\mathfrak N_{\varphi}$ and $\{ x_\varphi\, ;\, x\in
\mathfrak N_{\varphi}\} =H_{\varphi}\,$.
\item[(iii)] $\mathfrak X=\mathfrak N_{\varphi}$ and
$\displaystyle \sup\limits_{\substack{ 0\leq b\in M \\ b\neq 0}}\frac{\| b\|}{\;\!\varphi (b)\;\!}
<+\infty\,$.
\item[(iv)] $\mathfrak X=\mathfrak N_{\varphi}$ and $M$ is finite-dimensional.
\end{itemize}
\end{corollary}

\hfill $\square\quad$

\medskip
\section{Appendix: Some examples}
\bigskip

In the present appendix we show that there exist linear and antilinear involutions
on a infinite dimensional Hilbert space, defined everywhere and unbounded.

Let $H$ be an infinite dimensional Hilbert space equipped with a Hamel basis 
$\big(\xi_\iota\big)_{\iota\in I}$. Then each $\xi\in H$ can be uniquely written as
\begin{equation*}
\xi =\sum_{\iota\in I}\lambda_\iota (\xi )\xi_\iota\,,
\end{equation*}
where all but finitely many coefficients $\lambda_\iota (\xi )$ vanish.
In other words, the set $F_{\xi}$ of all indices $\iota\in I$ with $\lambda_\iota (\xi)
\neq 0$ is finite.

The linear functionals $\big(\lambda_\iota\big)_{\iota\in I}$ are called the
{\it coordinate functionals} of the Hamel basis $\big(\xi_\iota\big)_{\iota\in I}$.
\begin{proposition}
\label {uunnoo}
For a fixed Hamel basis $\big(\lambda_\iota\big)_{\iota\in I}$, 
$\left|\{\iota\in I\mid\lambda_\iota\,\,{\rm continuous}\,\,\}\right|<+\infty$.
\end{proposition}
\proof
Suppose that $\big|\{\iota\in I\mid\lambda_\iota\,\,\text{continuous}\,\,\}\big|=+\infty$.
Without loss of generality, we can suppose that the Hamel basis under consideration
is made of unit vectors. Choose a sequence $\big(\xi_{\iota_n}\big)_{n\geq 1}
\subset\big(\xi_\iota\big)_{\iota\in I}$ such that all the corresponding coordinate
functions $\big(\lambda_{\iota_n}\big)_{n\geq 1}$ are continuous. The sum
\medskip

\centerline{$\displaystyle \xi :=\sum_{n=1}^{+\infty}\frac1{2^n}\xi_{\iota_n}$}
\smallskip

\noindent is a well defined element of $H$
and
\medskip

\centerline{$\displaystyle \sum_{n=1}^{+\infty}\frac1{2^n}\xi_{\iota_n}=\xi =
\sum_{\iota\in F_{\xi}}\lambda_\iota(\xi )\xi_\iota\, .$}
\smallskip

\noindent By the finiteness of $F_{\xi}$ there exists $n_o\geq 1$ such that 
$$
\{\iota_{n_o},\iota_{n_o+1},\dots\}\cap F_{\xi}=\emptyset\,.
$$

Now the sequence
\smallskip

\centerline{$\displaystyle {\eta}_k:=\sum_{n=1}^{k}\frac1{2^n}\xi_{\iota_n}
-\xi = \sum_{n=1}^{k}\frac1{2^n}\xi_{\iota_n}
-\sum_{n=1}^{+\infty}\frac1{2^n}\xi_{\iota_n}
=-\sum_{n=k+1}^{+\infty}\frac1{2^n}\xi_{\iota_n}$}
\smallskip

\noindent is clearly convergent to $0$ and by the continuity of
$\lambda_{\iota_{n_o}}$ we obtain
\begin{equation}\label{zerolimit}
\lim_{k\to\infty}\lambda_{\iota_{n_o}}({\eta}_k)=\lambda_{\iota_{n_o}}
\big(\lim_{k\to\infty} {\eta}_k\big)=\lambda_{\iota_{n_o}}(0)=0\, .
\end{equation}
But for $k>n_o$ we have
\smallskip

\centerline{$\displaystyle \lambda_{\iota_{n_o}}({\eta}_k)=
\sum_{n=1}^{k}\frac1{2^n}\lambda_{\iota_{n_o}}(\xi_{\iota_n}) -
\lambda_{\iota_{n_o}}(\xi ) =\frac1{2^{n_o}}$}
\smallskip

\noindent because $\iota_{n_o}\notin F_{\xi}\,$. Consequently
\medskip

\centerline{$\displaystyle \lim_{k\to\infty}\lambda_{\iota_{n_o}}({\eta}_k)=
\frac1{2^{n_o}}\, ,$}
\smallskip

\noindent in contradiction to (\ref{zerolimit}).

\hfill $\square\quad$
\medskip

It is immediate to argue that, for each $n\geq 1\,$, any
infinite-dimensional Hilbert space $H$ admits a Hamel basis with at least
$n$ continuous coordinate functions and infinitely many non continuous
coordinate functions. It is indeed enough to split  $H$ in the orthogonal sum
of an $n$-dimensional linear subspace and its orthogonal complement,
which is infinite dimensional, and apply Proposition \ref{uunnoo} to the latter.
\begin{proposition}
Let $H$ be an infinite dimensional Hilbert space. Then there exists an
unbounded linear or antilinear involution $($i.e. with square equal to the identity
operator on $H)\; C : H\longrightarrow H$.
\end{proposition}

\proof
As explained below, there exists a Hamel basis $\big(\xi_\iota\big)_{\iota\in I}$
of $H$ admitting a continuous coordinate function $\lambda_{\iota_1}$, and a
non continuous one $\lambda_{\iota_2}$, for some $\iota_1 ,\iota_2$ in $I$.

Define the linear operator $C_1 : H\longrightarrow H$ by
\begin{equation}
\label{ciouno}
C_1\;\!\xi :=\sum_{\{\iota\in I\mid\iota\neq \iota_1 ,\iota_2\}}
\lambda_{\iota}(\xi )\xi_\iota+
\lambda_{\iota_2}(\xi )\xi_{\iota_1}+\lambda_{\iota_1}(\xi )\xi_{\iota_2}\, .
\end{equation}
If $C_1$ were continuous, then the linear functional defined by
$$
f(\xi ):=\frac{\big(\;\! \xi -C_1\;\!\xi\;\! \big|\;\!\xi_{\iota_1}-\xi_{\iota_2}\big)}{
\|\;\!\xi_{\iota_1}-\xi_{\iota_2}\|^2} =\lambda_{\iota_1}(x)-\lambda_{\iota_2}(x)
$$
would also be continuous, but it is not. Thus $C_1$ is the linear involution
we were searching for.

Concerning the antilinear case, either the antilinear operator
$C_2 : H\longrightarrow H$ defined by
\begin{equation}
\label{ciodue}
C_2\;\!\xi:=\sum_{\iota\in I}\,\overline{\lambda_{\iota}(\xi )}\xi_\iota
\end{equation}
is not continuous and we are done, or it is continuous. In the latter, the
operator
\medskip

\centerline{$C:=C_1C_2\, ,$}
\smallskip

\noindent where $C_1$ is the non continuous linear involution given
by \eqref{ciouno}, while $C_2$ is the continuous antilinear involution given
by \eqref{ciodue}, is antilinear and non continuous. It is also an involution
as $C_1C_2=C_2C_1$.

\hfill $\square\quad$

\bigskip


\begin{thebibliography}{99}

\bibitem{A} H. Araki: Some properties of modular conjugation operator of
von Neumann algebras and a non-commutative Radon-Nikodym theorem
with a chain rule,
\emph{Pacific J. Math.} \textbf{50} (1974), 309-354

\bibitem{A-Z} H. Araki, L.  Zsid\'{o}: Extension of the structure theorem of
Borchers and its application to half-sided modular inclusions,
\emph{Reviews in Math. Physics} \textbf{17} (2005), 491-543


\bibitem{Au1} B. Aupetit: The uniqueness of the complete norm topology
in Banach algebras and Banach Jordan algebras,
\emph{J. Funct. Analysis} \textbf{47} (1982), 1-6

\bibitem{Bon-Du} F. F. Bonsall, J. Duncan: \emph{Complete Normed
Algebras}, Springer-Verlag, 1973


\bibitem{C-Z2} I. Cior\u anescu, L. Zsid\'o: $\omega$-ultradistributions
and their application to operator theory,
in \emph{Spectral Theory $($Warsaw, 1977$)$}, Banach Center Publ.,
Vol. 8, Warsaw, 1982, 77-220


\bibitem{Co2} A. Connes: Caract\'erisation des espaces vectoriels
ordonn\'es sous-jacents aux alg\`ebres de von Neumann,
\emph{Ann. Inst. Fourier $($Grenoble$)$} \textbf{24} (1974), 121-155


\bibitem{Dix} J. Dixmier: \emph{Les alg\`ebres d'op\'eratours dans l'espace
Hilbertien $($Alg\`ebres de von Neumann$)$}, deuxi\`eme \'edition,
Gauthier-Villars, Paris, 1969

\bibitem{Dix-Ma} J. Dixmier, O. Mar\'echal: Vecteurs totalizateurs
d'une alg\`ebre de von Neumann,
\emph{Commun. Math. Phys.} \textbf{22} (1971), 44-50

\bibitem{Gr1} E. L. Griffin: Some contributions to the theory of rings of
operators II., \emph{Trans. Amer. Math. Soc.} \textbf{79} (1955), 389-400

\bibitem{Gr2} E. L. Griffin: Everywhere defined linear transformations
affiliated with rings of operators, \emph{Pacific J. Math.} \textbf{18} (1966),
489-493

\bibitem{Haa1} U. Haagerup: The standard form of von Neumann
algebras, \emph{Math. Scand.} \textbf{37} (1975), 271-283

\bibitem{Jo} B. E. Johnson: The uniqueness of the (complete) norm topology,
\emph{Bull. Amer. Math. Soc.} \textbf{73} (1967), 407-409

\bibitem{KR} R. V. Kadison, J. R. Ringrose: \emph{Fundamentals of the
Theory of Operator Algebras, Vol. IV. Special Topics. Advanced Theory -
An Exercise Approach},
Birkh\"auser, Boston, 1992

\bibitem{K} K. Knopp: \emph{Theory and Application of Infinite Series},
Blackie \& Son, London and Glasgow, 1951

\bibitem{O} T. Okayasu: A structure theorem of automorphisms of von
Neumann algebras,
\emph{T\^ohoku Math. J.} \textbf{20} (1968), 199-206

\bibitem{P} G. K. Pedersen: \emph{C*-Algebras and their Automorphism Groups},
Academic Press, 1979

\bibitem{P-T} G. K. Pedersen, M. Takesaki: The Radon-Nikodym
theorem for von Neumann algebras, \emph{Acta Math.} \textbf{130} (1973),
53-88

\bibitem{Sak1} S. Sakai: Weakly Compact Operators on Operator Algebras,
\emph{Pacific J. Math.} \textbf{14} (1964), 659-664

\bibitem{Sak2} S. Sakai: \emph{$C^*$-Algebras and $W^*$-Algebras},
Springer-Verlag, 1971

\bibitem{Sk} C. F. Skau: Finite Subalgebras of a von Neumann
Algebra, \emph{J. Funct. Analysis} \textbf{25} (1977), 211-235


\bibitem{S-Z} \c S. Str\u atil\u a, L. Zsid\'o: \emph{Lectures
on von Neumann Algebras}, Editura Academiei - Abacus Press, 1979

\bibitem{SZOA}  \c S. Str\u{a}til\u{a}, L.  Zsid\'{o}: \emph{Operator
Algebras}, INCREST Prepublication (1977-1979), 511 pages, to appear
at The Theta Foundation, Bucure\c{s}ti, Romania

\bibitem{T1} M. Takesaki: \emph{Tomita's Theory of Modular
Hilbert Algebras and its Applications}, Lecture Notes in Math.,
128, Springer-Verlag, 1970


\bibitem{Wil} A. Wilansky: \emph{Modern Methods in Topological Vector
Spaces}, McGraw-Hill, New York, 1978

\bibitem{Z1} L. Zsid\'o: Analytic generator and the foundation of
the Tomita-Takesaki theory of Hilbert algebras, in ``Proc.
International School Math. Phys.'', Univ. Camerino, 1974, 182-267

%


\end{thebibliography}
\end{document}